\documentclass[adraft,copyright,creativecommons, noderivs, noncommercial]{eptcs}

\usepackage{iftex}
\ifpdf
  \usepackage{underscore}
  \usepackage[T1]{fontenc}
\fi

\usepackage{amsmath,amsfonts}
\DeclareSymbolFont{AMSa}{U}{msa}{m}{n}
\DeclareMathSymbol{\twoheadrightarrow}{\mathrel}{AMSa}{"10}
\let\coprod\relax
\DeclareMathSymbol{\coprod}{\mathop}{AMSa}{"60}
\DeclareSymbolFont{AMSb}{U}{msb}{m}{n}
\DeclareMathSymbol{\varnothing}{\mathord}{AMSb}{"3F}
\usepackage{amsthm}
\usepackage{tikz-cd}
\tikzcdset{row sep/normal=0.7cm, column sep/normal=0.7cm}
\usepackage{mathtools}

\newtheorem{theo}{Theorem}[section]
\newtheorem{theorem}[theo]{Theorem}
\newtheorem{prop}[theo]{Proposition}

\newtheorem{coro}[theo]{Corollary}

\theoremstyle{definition}
\newtheorem{defi}[theo]{Definition}

\newtheorem{ex}[theo]{Example}

\newtheorem{contre-ex}[theo]{Counterexample}

\newtheorem{constr}[theo]{Construction}

\newtheorem{conv}[theo]{Convention}

\theoremstyle{remark}
\newtheorem{rem}[theo]{Remark}

\newcommand{\CC}{\mathcal{C}}
\newcommand{\DD}{\mathcal{D}}
\newcommand{\FF}{\mathcal{F}}

\newcommand{\cS}{\mathcal{S}}

\newcommand{\M}{\mathrm{M}}

\DeclareMathOperator{\Rr}{\mathbb{R}}

\renewcommand{\phi}{\varphi}

\DeclareMathOperator{\Coalg}{Coalg}
\newcommand{\Sys}{\mathrm{Sys}}
\DeclareMathOperator{\Beh}{Beh}

\DeclareMathOperator{\colim}{colim}
\newcommand{\id}{\mathrm{id}}
\newcommand{\Set}{\mathbf{Set}}

\title{Separation and Gluing of Explanations\\on Sites of Dynamical Systems}
\author{Paul Wang
\institute{LIP6, Sorbonne Universit\'{e}, Paris, France}
\email{paul.wang@lip6.fr}
}
\def\titlerunning{Separation and Gluing of Explanations on Sites of Dynamical Systems}
\def\authorrunning{P.~Wang}

\begin{document}
\maketitle

\begin{abstract}
We construct a Grothendieck site whose objects are Mealy machines over definable sets in an o-minimal structure and whose coverings are jointly surjective families of definable open immersions.
On this site, we define presheaves of \emph{explanations}---systems equipped with an interpretable interface, parameterised by a ``judge.''
We prove that the behavioral presheaf (quotienting by observable output equivalence) is \emph{separated}: a global explanation is determined by its local restrictions.
We show that gluing fails in general---locally consistent explanations need not assemble globally---and give, for stateless explanatory systems of the restricted-interface presheaf, a necessary and sufficient topological condition for the sheaf property in terms of robust disconnection of fibers of the judge.
\end{abstract}

\section{Introduction}\label{sec:intro}

Interpretability in machine learning seeks to explain \emph{why} a model produces a given output.
A striking feature of current methods is their inherent \textbf{locality}: explanations are given for specific data points, specific circuits, or specific feature directions---never for the model as a whole.
The fundamental question is: \emph{when multiple local explanations are available, do they fit together into a coherent global picture?}

This is precisely the question addressed by sheaf theory.
A sheaf assigns data to each open set and specifies how data on overlaps must agree; the sheaf condition captures the principle that compatible local data glue uniquely to global data.
When local data fail to glue, the obstruction can in principle be measured by sheaf cohomology (\S\ref{sec:discussion}).

Sheaf-theoretic approaches to dynamical systems have been studied before: Schultz, Spivak, and Vasilakopoulou~\cite{SSV2020} model systems as sheaves on time intervals; Goguen~\cite{Goguen1992} uses sheaves on observations; Joyal, Nielsen, and Winskel~\cite{JNW1996} use presheaves on paths.
In each case, a Grothendieck topology is placed on a ``base'' category, and systems are modelled \emph{as}~sheaves.
In a different vein, Abramsky and Brandenburger~\cite{AbramskyBrandenburger2011} use sheaves to characterise non-locality and contextuality in quantum foundations: local measurement data fail to extend to a global section, an obstruction measured by sheaf cohomology.
Our gluing failure (Counterexample~\ref{cex:beh-gluing-fail}) is structurally analogous---compatible local explanations fail to extend globally---but the mathematical content is entirely different: their presheaves live on measurement scenarios, ours on categories of dynamical systems.
On the interpretability side, Geiger et al.~\cite{Geiger_etal2025} develop \emph{causal abstraction} as a framework for mechanistic interpretability; their notion of faithfulness relates a single abstraction to a model, whereas our framework studies when \emph{multiple} local abstractions (sections) compose.
Bassan, Amir, and Katz~\cite{Bassan_etal2024} study local-vs-global interpretability from a computational complexity perspective.

We take a different approach: we define a Grothendieck topology on a category whose \emph{objects are systems} and whose morphisms are open immersions.
The sheaf condition on such a site formalises a local-to-global principle with respect to \emph{system representations}.
To the best of our knowledge, no such site structure has previously been studied.

\paragraph{Contributions.}
Working over definable sets in an o-minimal structure (a setting providing tame geometry and a canonical class of open immersions):
\begin{enumerate}
\item We prove that categories of homogeneous Mealy machines over an adhesive base inherit adhesivity (Theorem~\ref{thm:adhesivity}), and construct a Grothendieck site $(\Sys_{\mathrm{oi}}, \mathrm{Cov})$ of definable open immersions (\S\ref{sec:systems-sites}).
\item We define a hierarchy of presheaves of explanations, parameterised by a ``judge'' translating the system's interface to interpretable concepts (\S\ref{sec:judges-presheaves}).
\item We prove that the behavioral presheaf $\FF_j^{\mathrm{beh}}$ is \emph{separated} (\S\ref{sec:separation}): a global explanation is determined by its local restrictions.
\item We show that gluing \emph{fails} in general (\S\ref{sec:gluing}), and give a necessary and sufficient topological condition for the sheaf property of the presheaf $\FF_{j,1}^{\mathrm{ri}}$ of restriced-interface stateless explanations \ref{subsec:positive}, in terms of robust disconnection of fibers of the judge (Theorem~\ref{thm:singleton-ns}).
\item We introduce $\varepsilon$-approximate sections (\S\ref{sec:approximate}) and show that the dimension of the output space controls the depth of higher-order reconciliation obstructions via Helly's theorem (Theorem~\ref{thm:helly}).
\end{enumerate}

\noindent\textbf{Notation.}
We write $\cS = (S_b, S_a, I, O, \alpha)$ for a heterogeneous Mealy machine with before-states $S_b$, after-states $S_a$, inputs $I$, outputs $O$, and dynamics $\alpha : S_b \times I \to S_a \times O$.
All systems have non-empty $I$ and $O$.

\section{Systems, Sites, and the O-Minimal Setting}\label{sec:systems-sites}

We set up the categorical framework: Mealy machines over definable sets in an o-minimal structure, definable open immersions as morphisms, and the resulting Grothendieck site.

\subsection{Heterogeneous Mealy machines}\label{subsec:mealy}

\begin{defi}\label{def:sys}
Let $\CC$ be a category with finite products.
A \textbf{heterogeneous Mealy machine} is a quintuple $\cS = (S_b, S_a, I, O, \alpha)$ with $\alpha : S_b \times I \to S_a \times O$, where $S_b$ (before-states) and $S_a$ (after-states) may differ; this arises naturally when restricting to a subset of states whose after-states may leave the subset (cf.~\cite{SprungerKatsumata2019} for a categorical treatment of time-varying state spaces).
When $S_b = S_a =: S$, we recover a \emph{homogeneous} Mealy machine $(S, I, O, \alpha)$; in $\Set$, these are coalgebras for the endofunctor $H(X) = (X \times O)^I$.
\end{defi}

A \textbf{morphism} $\cS \to \cS'$ is a 4-tuple $(f_b, f_a, f_I, f_O)$ making the dynamics square commute:
\[
\begin{tikzcd}
S_b \times I \arrow[r, "\alpha"] \arrow[d, "f_b \times f_I"'] & S_a \times O \arrow[d, "f_a \times f_O"] \\
S_b' \times I' \arrow[r, "\alpha'"'] & S_a' \times O'
\end{tikzcd}
\]
We write $\Sys$ for the resulting category, and $\Sys_{\mathrm{ho}}(I,O)$ for the \emph{non full} subcategory of fixed-$(I,O)$ homogeneous systems where morphisms between state objects are the same for before-state and after-state.

\subsection{The o-minimal setting}\label{subsec:o-minimal}

ACT approaches to dynamical systems typically work over $\Set$, topological spaces, or categories of smooth manifolds.
We work instead over $\mathrm{Def}(\M)$, the category of \emph{definable sets} in an o-minimal structure $\M$, or its wide subcategory $\mathrm{Def}_c(\M)$ of continuous definable maps (Remark~\ref{rem:why-not-set}).
The motivation comes from a classical idea of Grothendieck, who proposed in his \emph{Esquisse d'un Programme}~\cite{Grothendieck1997} that the category of topological spaces is too wild for geometry, and called for a theory of \emph{topologie mod\'er\'ee} (tame topology):

\begin{quote}
\emph{%
La notion de ``espace topologique'' [...] est manifestement inad\'equate [...].
La structure topologique [...] laisse coexister des ph\'enom\`enes ``mod\'er\'es'' [...] avec des ph\'enom\`enes ``sauvages'' [...].
C'est la pr\'edominance de ces derniers qui rend inad\'equate la notion actuelle de ``vari\'et\'e topologique'' [...].
On voit mal d'ailleurs quelle id\'ee raisonnable on pourrait se faire d'un ``espace mod\'er\'e'' de dimension quelconque, tant qu'on ne dispose pas d'un ``principe de mod\'eration'' qui permette de se mettre \`a l'abri.%
}
\hfill---Grothendieck, \emph{Esquisse d'un Programme}~\cite{Grothendieck1997}, \S5
\end{quote}

O-minimal structures, introduced by van den Dries~\cite{vandenDries1998} following work of Pillay and Steinhorn, provide precisely such a ``principe de mod\'eration'': a single axiom on one-dimensional definable sets forces tameness in all dimensions.

\begin{defi}[{\cite[Ch.~1]{vandenDries1998}}]\label{def:o-minimal}
An \textbf{o-minimal structure} $\M$ expanding $(\Rr, <, +, \cdot)$ is given by a sequence $(\DD_n)_{n \geq 1}$ of Boolean algebras of subsets of $\Rr^n$, closed under products and projections, such that $\DD_n$ contains all the semialgebraic sets and $\DD_1$ consists exactly of finite unions of intervals and points.
Elements of $\DD_n$ are \textbf{definable sets}; a function between them is \textbf{definable} if its graph is.
\end{defi}

The condition on $\DD_1$ is the ``principle of moderation'': it forbids Cantor sets, space-filling curves, and other pathologies.
The key consequences for our purposes are:

\begin{enumerate}
\item \emph{Cell decomposition.} Every definable set $X \subseteq \Rr^n$ admits a finite partition into \emph{cells} (definable sets of a simple form, that are canonically homeomorphic to open balls), and every definable function on $X$ can be made piecewise continuous by refining the partition~\cite[Ch.~3]{vandenDries1998}.
In particular, definable open subsets are finite unions of connected definable opens.

\item \emph{Adhesivity.} The category $\mathrm{Def}(\M)$ is adhesive~\cite{LackSobocinski2005}: it has pullbacks, pushouts along monomorphisms, and these satisfy the Van Kampen condition.
The proof uses the faithful embedding $\mathrm{Def}(\M) \hookrightarrow \Set$: pullbacks and pushouts along monomorphisms are computed on underlying sets and remain definable (images of definable maps are definable), so the Van Kampen condition reduces to the known adhesivity of $\Set$.
Adhesivity ensures that pushouts along open immersions preserve monomorphisms---the condition will be used for cogerm gluing (Theorem~\ref{thm:gluing-Fj}).

\item \emph{Finite cohomology.} Edmundo and Peatfield~\cite{EdmundoPeatfield2008} develop \v{C}ech cohomology for definable sheaves over o-minimal structures, with finite admissible covers and finite simplicial nerves.
Our site $\Sys_{\mathrm{oi}}$ is not directly a definable site in their sense (it is a category of systems, not a single definable space), but coverings of a fixed system have finite definable nerves, making their techniques a natural starting point for the cohomological obstructions discussed in \S\ref{sec:discussion}.
\end{enumerate}

Each definable set $X \subseteq \Rr^n$ inherits a topology from $\Rr^n$.
A \textbf{definable open immersion} $\iota : U \hookrightarrow X$ is a definable injection whose image is a definable open subset.

\begin{ex}\label{ex:o-minimal-examples}
Standard examples of o-minimal structures: $\Rr_{\mathrm{alg}}$ (semialgebraic sets---polynomials only), $\Rr_{\mathrm{an}}$ (globally subanalytic sets---convergent power series), $\Rr_{\mathrm{an},\exp}$ (adding the exponential function).
Neural network activation functions (ReLU, sigmoid, softmax) are definable in $\Rr_{\mathrm{an},\exp}$, so they naturally fit in this framework.
Indeed, o-minimality has found direct applications in deep learning theory: Ji and Telgarsky~\cite{JiTelgarsky2020} use tools from o-minimal analysis to prove directional convergence of gradient flow on deep homogeneous networks, and Bolte and Pauwels~\cite{BoltePauwels2021} exploit Whitney stratification of definable sets for convergence in values of stochastic gradient methods.
\end{ex}

\begin{rem}\label{rem:why-not-set}
Working over $\Set$ would allow all injections as morphisms, with no topological control: any partition of $S_b$ gives a covering, and the geometry of the state space plays no role.
In $\mathrm{Def}(\M)$, the o-minimal topology provides a canonical class of ``local'' inclusions, finite coverings with finitely many connected components, and the tameness needed for the sheaf characterisation of Theorem~\ref{thm:singleton-ns}.
The cost is that $\mathrm{Def}(\M)$ is not cartesian closed.

Morphisms in $\mathrm{Def}(\M)$ need not be continuous; we write $\mathrm{Def}_c(\M)$ for the wide subcategory with continuous definable maps.
The site construction works over both $\mathrm{Def}(\M)$ and $\mathrm{Def}_c(\M)$, since open immersions are continuous by nature (inclusions of open subsets with the subspace topology).
Adhesivity, separation, and cogerm gluing use only $\mathrm{Def}(\M)$.
The sheaf characterisation (Theorem~\ref{thm:singleton-ns}) requires $\mathrm{Def}_c(\M)$: its proof uses compactness and normality of the input space, which depend on maps being continuous.
\end{rem}

\subsection{The site $\Sys_{\mathrm{oi}}$}\label{subsec:site}

We now combine the system structure with the o-minimal geometry.

\begin{defi}[Definable open immersion of systems]\label{def:open-immersion}
A morphism $\iota : \cS' \hookrightarrow \cS$ in $\Sys$ (over $\mathrm{Def}(\M)$) is called a \textbf{definable open immersion} if each component $\iota_b, \iota_a, \iota_I, \iota_O$ is a definable open immersion.
Write $\Sys_{\mathrm{oi}}$ (``open immersions'') for the wide subcategory with these as morphisms.
\end{defi}

\begin{defi}[Coverings]\label{def:coverings}
A finite family $\{m_\alpha : \cS_\alpha \hookrightarrow \cS\}$ of definable open immersions is a \textbf{covering} if the induced maps are jointly surjective:
$\bigsqcup_\alpha (S_{\alpha,b} \times I_\alpha) \twoheadrightarrow S_b \times I$
and
$\bigsqcup_\alpha (S_{\alpha,a} \times O_\alpha) \twoheadrightarrow S_a \times O$.
\end{defi}

Joint surjectivity on $S_b \times I$ is strictly stronger than separate-factor surjectivity ($\bigsqcup S_{\alpha,b} \twoheadrightarrow S_b$ and $\bigsqcup I_\alpha \twoheadrightarrow I$ separately): the product condition ensures that for every $(s, i) \in S_b \times I$, some single patch covers both $s$ and $i$.

\begin{prop}\label{prop:site-axioms}
$\Sys_{\mathrm{oi}}$ has pullbacks (computed componentwise as intersections of definable opens), and the coverings above define a Grothendieck pretopology~\cite[Ch.~III]{MacLaneMoerdijk1994}, making $(\Sys_{\mathrm{oi}}, \mathrm{Cov})$ a site.
\end{prop}

\begin{proof}[Proof sketch]
(T1)~Isomorphisms cover trivially.\par\noindent
(T2)~Given a covering $\{m_\alpha\}$ of $\cS$ and $n : \cS' \hookrightarrow \cS$, the pullback has components $(S_{\alpha,b} \cap S_b') \times (I_\alpha \cap I')$; joint surjectivity restricts from $S_b \times I$ to $S_b' \times I'$.\par\noindent
(T3)~Compositions of definable open immersions are definable open immersions, and joint surjectivity is transitive.
\end{proof}

\begin{theorem}\label{thm:adhesivity}
  For all $I$, $O$ in $\mathrm{Def}(\M)$, the category $\Sys_{\mathrm{ho}}(I,O)$ is adhesive.
\end{theorem}

This is a special case of a more general result (Appendix~\ref{app:adhesivity}):

\begin{theorem}\label{thm:coalg_adhesivity_body}
Let $\CC$ be an adhesive category~\cite{LackSobocinski2005} with finite products, and let $H : \CC \to \CC$ be a pullback-preserving endofunctor. Then the category of coalgebras $\Coalg(H)$ is adhesive.
\end{theorem}

For $\CC = \Set$, this was established by Padberg~\cite{Padberg2017}; the general case appears to be new.
Adhesivity ensures colimit constructions in gluing proofs produce well-defined homogeneous systems; combined with the strict initial object (the empty definable set), it also gives extensivity~\cite[Lemma~4.1]{LackSobocinski2005}.

\section{Judges and Presheaves of Explanations}\label{sec:judges-presheaves}

An \emph{explanation} of a system translates its dynamics into human-interpretable terms.
Here, we model human-interpretability as being relative to a \emph{judge}: a choice of interpretable interface, along with a map from the system's interface to the interpretable one.

\subsection{Judges and sections}\label{subsec:judges}

\begin{defi}[Judge]\label{def:judge}
A \textbf{judge} for $\cS = (S_b, S_a, I, O, \alpha)$ is a pair $j = (j_I, j_O)$ of definable maps $j_I : I \to I'$, $j_O : O \to O'$, where $(I', O')$ is the \textbf{interpretable interface}.
\end{defi}

\begin{constr}[Presheaf of explanations]\label{constr:presheaf}
Fix a judge $j : (I, O) \to (I', O')$.
For a definable open immersion $m : \cS_U \hookrightarrow \cS$, define $\tilde{\FF}_j(\cS_U)$ as the set of pairs $(\cS', \psi)$ where $\cS' = (S_b', S_a', I', O', \alpha')$ and $\psi : \cS_U \to \cS'$ is a morphism with interface components $\psi_I = j_I \circ m_I$ and $\psi_O = j_O \circ m_O$.
The dynamics square
\[
\begin{tikzcd}
S_{U,b} \times I_U \arrow[r, "\alpha_U"] \arrow[d, "\psi_b \times (j_I \circ m_I)"'] & S_{U,a} \times O_U \arrow[d, "\psi_a \times (j_O \circ m_O)"] \\
S_b' \times I' \arrow[r, "\alpha'"'] & S_a' \times O'
\end{tikzcd}
\]
says that $\cS'$ faithfully represents the dynamics of $\cS_U$ through the judge.
Restriction is by precomposition: $(\cS', \psi) \mapsto (\cS', \psi \circ n)$ for $n : \cS_{U'} \hookrightarrow \cS_U$.
\end{constr}

The presheaf $\tilde{\FF}_j$ on the slice site $(\Sys/\cS)_{\mathrm{oi}}^{\mathrm{op}}$ is a sheaf (separation and gluing reduce to joint surjectivity, since sections are literal pairs $(\cS', \psi)$), but too fine: it distinguishes explanations that differ only in inessential internal structure.
The question is which quotient to take.

\subsection{Equivalence relations and the presheaf hierarchy}\label{subsec:hierarchy}

\begin{conv}\label{conv:homogeneous}
Throughout the remainder of this paper, all \textbf{explanatory systems} $\cS'$ (targets of sections) are homogeneous ($S_b' = S_a' =: S'$), so that dynamics can be iterated and behavioral equivalence is well-defined.
The systems being explained remain heterogeneous, for they can be given by patches of before-states and after-states that don't coincide.
\end{conv}

\begin{defi}[Cogerm equivalence]\label{def:cogerm}
Two sections $(\cS_1', \psi_1)$ and $(\cS_2', \psi_2)$ over $\cS_U$ are \textbf{cogerm-equivalent} if there exists a homogeneous system $\cS''$ with interface $(I', O')$, monomorphisms $i_k : \cS'' \hookrightarrow \cS_k'$ ($k=1,2$) in $\Sys_{\mathrm{ho}}(I', O')$ (i.e.\ with identity interface components and $i_{k,b} = i_{k,a}$), and a~$\phi : \cS_U \to \cS''$ (necessarily unique) with $i_k \circ \phi = \psi_k$.
Write $(\cS_1', \psi_1) \sim (\cS_2', \psi_2)$.
\end{defi}

The name reflects an analogy with germs: two explanations agree if they share a common ``essential core'' on the target side.
The structure is parallel to the span-based characterisation of bisimulation in Joyal--Nielsen--Winskel~\cite{JNW1996}: both use spans as equivalence witnesses, but with monomorphisms in place of open maps and on the target side rather than the path side.

The quotient $\FF_j := \tilde{\FF}_j / {\sim}$ satisfies gluing but not separation.

\begin{defi}[Behavioral equivalence]\label{def:beh-equiv}
Two explanations $(\cS_1', \psi_1)$ and $(\cS_2', \psi_2)$ over $\cS_U$ are \textbf{behaviorally equivalent}, written $\approx$, if for every $x \in S_{U,b}$, every $n$, and every $(i_1', \ldots, i_n') \in (I')^n$, the output sequences from $\psi_{1,b}(x)$ in $\cS_1'$ and $\psi_{2,b}(x)$ in $\cS_2'$ coincide (well-defined since $\cS_k'$ is homogeneous by Convention~\ref{conv:homogeneous}).
\end{defi}

Cogerm equivalence implies behavioral equivalence, but not conversely (e.g., a $3$-cycle and a $6$-cycle with the same constant output are behaviorally equivalent but not cogerm-equivalent).
For homogeneous Mealy machines over $\Set$, behavioral equivalence is related to the standard coalgebraic notion of equality of images in the final coalgebra for $H(X) = (X \times O')^{I'}$~\cite{Rutten_2000}.
From the interpretability perspective, behavioral equivalence is a natural relation: an explanation serves to predict behavior through the judge, and two explanations making identical predictions are interchangeable.

\begin{defi}[Restricted-interface presheaf]\label{def:ri}
For $m : \cS_U \hookrightarrow \cS$, define $I'_U := \mathrm{im}(j_I \circ m_I) \subseteq I'$.
The \textbf{restricted-interface presheaf} $\FF_j^{\mathrm{ri}}(\cS_U)$ consists of explanations $(\cS', \psi)$ where $\cS'$ has input interface $I'_U$ (rather than all of $I'$), modulo restricted-behavioral equivalence (testing input sequences in $(I'_U)^n$ only).
\end{defi}

The natural transformation $\eta : \FF_j^{\mathrm{beh}} \to \FF_j^{\mathrm{ri}}$ given by restricting the input interface is surjective on sections (any section of $\FF_j^{\mathrm{ri}}$ over $\cS_U$ extends to one of $\FF_j^{\mathrm{beh}}$ by choosing arbitrary dynamics on $I' \setminus I'_U$, which does not affect the restricted-behavioral class) and an isomorphism when $I'_U = I'$.
In practice, the distinction maps onto two modes of interpretability: sections of $\FF_j^{\mathrm{beh}}$ would correspond to \emph{mechanistic explanations} (``this feature detects X \emph{across all inputs}''), while sections of $\FF_j^{\mathrm{ri}}$ correspond to \emph{local explanations} (``this circuit implements Y \emph{on distribution $\mathcal{D}$}'').

We write $\FF_j^{\mathrm{beh}} := \tilde{\FF}_j / {\approx}$ for the \textbf{behavioral presheaf}.
The hierarchy is:
\[
\tilde{\FF}_j \twoheadrightarrow \FF_j \twoheadrightarrow \FF_j^{\mathrm{beh}} \overset{\eta}{\twoheadrightarrow} \FF_j^{\mathrm{ri}}.
\]
$\tilde{\FF}_j$ is a sheaf but too fine; the three quotients capture progressively coarser notions of explanatory content.
None of the quotients is a sheaf (\S\ref{sec:separation}--\ref{sec:gluing}).

\section{Separation}\label{sec:separation}

The separation axiom asks: if two global sections restrict to the same class on every member of a covering, are they globally equivalent?

\begin{theorem}[Behavioral presheaf is separated]\label{thm:separation}
$\FF_j^{\mathrm{beh}}$ satisfies the separation axiom on $(\Sys_{\mathrm{oi}}, \mathrm{Cov})$.
\end{theorem}

\begin{proof}
Behavioral equivalence is pointwise: $(\cS_1', \psi_1) \approx (\cS_2', \psi_2)$ iff $\Beh_{\cS_1'}(\psi_{1,b}(x)) = \Beh_{\cS_2'}(\psi_{2,b}(x))$ for every $x \in S_{U,b}$, where $\Beh$ denotes the map from states to infinite output trees.
If this holds on every $\cS_\alpha$ in a jointly surjective covering, then for every $x \in S_{U,b}$ there exists $\alpha$ with $x \in S_{\alpha,b}$, so the behaviours agree at $x$.
\end{proof}

The cogerm presheaf $\FF_j$ does \emph{not} satisfy separation: local cogerm witnesses may contain ``extra states'' outside the image of the section, and these extra states can disagree across patches.
The passage to the behavioral quotient eliminates this problem by testing only observable output sequences.

\begin{defi}[$j$-full covering]\label{def:j-full}
A covering $\{m_\alpha : \cS_\alpha \hookrightarrow \cS\}$ is \textbf{$j$-full} if $I'_{U_\alpha} = I'_U$ for every $\alpha$, i.e.\ each patch's inputs hit every fiber of $j_I$.
Coverings where $m_{I,\alpha} = \id_I$, which we call \textbf{data-local}, are automatically $j$-full.
\end{defi}

\begin{prop}[Separation for $\FF_j^{\mathrm{ri}}$]\label{prop:ri-separation}
$\FF_j^{\mathrm{ri}}$ satisfies separation for any $j$-full covering.
\end{prop}

\begin{proof}
When $I'_{U_\alpha} = I'_U$ for all $\alpha$, local and global restricted-behavioral equivalences test the same input sequences $(I'_U)^n$.
Joint surjectivity on $S_b$ (by projection from $S_b \times I$) completes the argument.
\end{proof}

Separation can fail when the covering splits inputs across patches:
\begin{contre-ex}\label{cex:ri-sep-fail}
Consider $\cS = (\{s_1, s_2\}, \{s_1, s_2\}, \{a,b\}, \{0,1\}, \alpha)$ with $\alpha(s_1, a) = (s_1, 0)$, $\alpha(s_1, b) = (s_2, 0)$, $\alpha(s_2, a) = (s_1, 1)$, $\alpha(s_2, b) = (s_2, 1)$, and $j = \id$.
The input-splitting covering $\cS_1 = \cS|_{\{a\}}$, $\cS_2 = \cS|_{\{b\}}$ has $I'_{U_1} = \{a\} \neq \{b\} = I'_{U_2}$.
An alternative explanation $\cS_2'$ with different internal routing agrees with the identity explanation on each patch's pure input sequences, but disagrees on the mixed sequence $(a,b)$.
\end{contre-ex}

\section{The Gluing Landscape}\label{sec:gluing}

While the cogerm presheaf $\FF_j$ always satisfies gluing (Theorem~\ref{thm:gluing-Fj}), the behavioral and restricted-interface presheaves do not.
We map out the landscape of positive and negative results.

\subsection{Gluing for $\FF_j$}\label{subsec:gluing-Fj}

\begin{theorem}\label{thm:gluing-Fj}
$\FF_j$ satisfies the gluing axiom: given a covering $\{m_\alpha : \cS_\alpha \hookrightarrow \cS\}$ and a compatible family of sections $[\cS_\alpha', \psi_\alpha] \in \FF_j(\cS_\alpha)$, there exists $[\cS', \psi] \in \FF_j(\cS)$ restricting to each $[\cS_\alpha', \psi_\alpha]$.
\end{theorem}

\begin{proof}[Proof sketch]
Choose cogerm witnesses $\cS_{\alpha\beta}''$ for compatibility on overlaps.
The global target $\cS'$ is the colimit of $\cdots \cS_\alpha' \xleftarrow{i_\alpha} \cS_{\alpha\beta}'' \xrightarrow{i_\beta} \cS_\beta' \cdots$ in $\Sys$; this exists as iterated pushouts along monomorphisms in the adhesive category $\mathrm{Def}(\M)$.
Adhesivity ensures the canonical maps $\cS_\alpha' \hookrightarrow \cS'$ remain monic; the compatible morphisms $\psi_\alpha$ then glue to $\psi : \cS \to \cS'$ by joint surjectivity.
\end{proof}

\subsection{Gluing fails for $\FF_j^{\mathrm{beh}}$}\label{subsec:gluing-fails}

\begin{contre-ex}[Gluing failure for $\FF_j^{\mathrm{beh}}$]\label{cex:beh-gluing-fail}
Let $\cS = (\{s_0, s_1, s_2, s_3\}, \{s_0, s_1, s_2, s_3\}, \{a,b\}, \{0,1\}, \alpha)$ with dynamics:
\begin{align*}
\alpha(s_0, a) &= (s_0, 0), & \alpha(s_0, b) &= (s_2, 0), \\
\alpha(s_1, a) &= (s_1, 1), & \alpha(s_1, b) &= (s_1, 1), \\
\alpha(s_2, a) &= (s_1, 1), & \alpha(s_2, b) &= (s_1, 1), \\
\alpha(s_3, a) &= (s_3, 1), & \alpha(s_3, b) &= (s_2, 1),
\end{align*}
with judge $j_I : \{a,b\} \to \{\bullet\}$, $j_O = \id$.
The covering $\cS_1 = \cS|_{\{s_0, s_1, s_2\}}$, $\cS_2 = \cS|_{\{s_1, s_2, s_3\}}$ has non-trivial overlap $\{s_1, s_2\}$ and is data-local.

Over $\cS_1$: the minimal realisation $M_1 = (\{p_0, p_1\}, \{\bullet\}, \{0,1\})$ maps $s_0 \mapsto p_0$ (constant output~$0$) and $s_1, s_2 \mapsto p_1$ (constant output~$1$).
Over $\cS_2$: the realisation $M_2 = (\{q\}, \{\bullet\}, \{0,1\})$ maps $s_1, s_2, s_3 \mapsto q$ (constant output~$1$).
Compatibility holds on $\{s_1, s_2\}$: both produce constant output~$1$.

A global section $(M, \bar{\psi})$ would require: $\bar{\psi}(s_0)$ producing output~$0$, and $\bar{\psi}(s_3)$ producing output~$1$.
The dynamics square at $(s_0, b)$ forces $\bar{\psi}_a(s_2)$ to continue with output~$0$; at $(s_3, b)$, it forces $\bar{\psi}_a(s_2)$ to continue with output~$1$.
Contradiction: a single after-state cannot produce both behaviours.
\end{contre-ex}

The obstruction is that the dynamics sends states from distinct behavioral regions ($s_0$ and $s_3$, in different patches) to a common after-state $s_2$.
Compatibility constrains the overlap $\{s_1, s_2\}$ but not the after-state reached from \emph{outside} the overlap.
The non-injectivity of $j_I$ is essential: collapsing inputs forces a single dynamics-square constraint per initial state.

\subsection{Positive results: connected and robustly disconnected fibers}\label{subsec:positive}

Write $\FF_{j,1}^{\mathrm{ri}} \subseteq \FF_j^{\mathrm{ri}}$ for the \textbf{stateless subpresheaf}: sections $(\cS', \psi)$ with $|S'| = 1$.
(Restriction preserves $|S'| = 1$, so this is indeed a subpresheaf; the notation generalises to $\FF_{j,N}^{\mathrm{ri}}$ for $|S'| \leq N$.)
A stateless section reduces to a function $h : I_U \to O'$ with $h = g \circ j_I$ for a uniquely determined $g$, and gluing becomes a local-to-global question about $\ker(j_I)$-invariance of continuous definable functions.
We give a precise characterisation in the o-minimal setting.

\begin{defi}[Robustly disconnected fiber]\label{def:robust-disconnection}
A fiber $j_I^{-1}(t_0)$ is \textbf{robustly disconnected} if there exist an open $N \ni t_0$ in $I'$ and disjoint non-empty opens $V, W$ in $I$ with $j_I^{-1}(N) \subseteq V \cup W$, $V \cap j_I^{-1}(t_0) \neq \varnothing$, $W \cap j_I^{-1}(t_0) \neq \varnothing$.
\end{defi}

Connected fibers are never robustly disconnected.
The converse is strict: a fiber can be disconnected but not \emph{robustly} so, when the disconnection ``heals'' in nearby fibers.

\begin{ex}[Punctured square]\label{ex:punctured-square}
Let $I = [0,1]^2 \setminus \{(0.5, 0.5)\}$, $j_I(x,y) = x$.
The fiber at $x = 0.5$ is disconnected, but nearby fibers $\{x\} \times [0,1]$ ($x \neq 0.5$) are connected and cannot be separated by $V \cup W$.
No robust disconnection exists, so $\FF_{j,1}^{\mathrm{ri}}$ is a sheaf despite the disconnected fiber.
\end{ex}

\begin{theorem}[Sheaf characterisation for stateless explanations]\label{thm:singleton-ns}
In $\mathrm{Def}_c(\M)$ (the wide subcategory with continuous definable maps), suppose $I$ is compact and $\mathrm{im}(j_O)$ is definably connected with $|\mathrm{im}(j_O)| \geq 2$.
Then
\[
\FF_{j,1}^{\mathrm{ri}} \text{ is a sheaf} \iff \text{no fiber of } j_I \text{ is robustly disconnected.}
\]
When $\mathrm{im}(j_O)$ is totally disconnected, the sheaf condition holds trivially.
\end{theorem}

\begin{proof}[Proof sketch]
($\Rightarrow$) Given a robust disconnection $j_I^{-1}(N) \subseteq V \cup W$, use the connected image of $j_O$ and a definable bump function to construct a continuous $f : I \to O$ that is $\ker(j_I)$-invariant on $V$ and on $W$ separately, but with $j_O(f(a)) \neq j_O(f(b))$ for $a \in V \cap j_I^{-1}(t_0)$, $b \in W \cap j_I^{-1}(t_0)$.
The covering $U_1 := V \cup (I \setminus j_I^{-1}(N))$, $U_2 := W \cup (I \setminus j_I^{-1}(N))$ provides compatible local sections that do not glue.

($\Leftarrow$) If $h$ is locally $\ker(j_I)$-invariant but $h(x) \neq h(y)$ with $j_I(x) = j_I(y) = t_0$, separate the level sets of $h|_{j_I^{-1}(t_0)}$ by disjoint opens (normality of compact $I$), then use compactness to find $N \ni t_0$ with $j_I^{-1}(N)$ contained in the union---a robust disconnection.
\end{proof}

\begin{coro}\label{cor:connected-fibers}
Connected fibers of $j_I$ $\Rightarrow$ no robust disconnection $\Rightarrow$ $\FF_{j,1}^{\mathrm{ri}}$ is a sheaf.
The first implication is strict (Example~\ref{ex:punctured-square}).
\end{coro}

\subsection{Summary of the gluing landscape}\label{subsec:landscape}

\begin{center}
\renewcommand{\arraystretch}{1.2}
\begin{tabular}{lcc}
\hline
\textbf{Presheaf} & \textbf{Separated?} & \textbf{Gluing?} \\
\hline
$\tilde{\FF}_j$ (unquotiented) & Yes & Yes (sheaf) \\
$\FF_j$ (cogerm) & No & Yes \\
$\FF_j^{\mathrm{beh}}$ (behavioral) & Yes & No (Cex.~\ref{cex:beh-gluing-fail}) \\
$\FF_j^{\mathrm{ri}}$ (restricted-interface) & $j$-full only & No in general \\
$\FF_{j,1}^{\mathrm{ri}}$ (stateless) & $j$-full only & Thm.~\ref{thm:singleton-ns}: iff no robust disconn. \\
\hline
\end{tabular}
\end{center}

\noindent The pattern: finer equivalence relations admit gluing but break separation; coarser ones yield separation but break gluing.
The ``right'' quotient depends on the interpretability task.

\section{Approximate Sections and Helly Obstructions}\label{sec:approximate}

The exact sheaf condition asks whether locally consistent explanations glue to a global one.
In practice, explanations are rarely exact: a probe achieves a given accuracy, a LIME fit has a fidelity bound, a concept bottleneck leaks information.
We relax the dynamics square from exact equality to $\varepsilon$-approximation and ask: when do locally \emph{approximately consistent} explanations compose into a globally approximately consistent one?

\subsection{$\varepsilon$-Sections}\label{subsec:eps-sections}

Let $(O', \rho)$ be a metric space.
For a stateless system $\cS_U$, i.e. where $|S_{U,b}| = |S_{U,a}| = 1$, write $f : I_U \to O$ for the output function $f(i) := \pi_O(\alpha_U(*, i))$, where $\pi_O$ is the projection onto~$O$.
A stateless explanatory section ($|S'| = 1$) over $\cS_U$ reduces to a function $g : I' \to O'$ satisfying the dynamics square $g(j_I(i)) = j_O(f(i))$ for all $i \in I_U$.

\begin{defi}[$\varepsilon$-section]\label{def:eps-section}
An \textbf{$\varepsilon$-section} over $\cS_U$ is a function $g : I' \to O'$ satisfying
\[
  \rho(g(j_I(i)),\; j_O(f(i))) \leq \varepsilon \qquad \text{for all } i \in I_U.
\]
Write $\FF_j^\varepsilon(\cS_U)$ for the set of $\varepsilon$-sections.
Note that $g$ is defined on the full interface~$I'$, not the restricted $I'_U := \mathrm{im}(j_I|_{I_U})$; this is the $\varepsilon$-relaxation of the stateless behavioral presheaf (stateless sections of $\FF_j^{\mathrm{beh}}$ with $|S'| = 1$), not of $\FF_{j,1}^{\mathrm{ri}}$.
At $\varepsilon = 0$, this recovers the exact stateless sections of $\FF_j^{\mathrm{beh}}$.
\end{defi}

For each interpretable input $i' \in I'_U$, define the \textbf{target set}
\[
T(i', U) := \{ j_O(f(i)) : i \in I_U,\; j_I(i) = i' \} \subseteq O'
\]
and the \textbf{feasibility set} $A_\varepsilon(i', U) := \bigcap_{t \in T(i', U)} \bar{B}(t, \varepsilon)$, the intersection of closed $\varepsilon$-balls centred at the target values.
An $\varepsilon$-section exists iff $A_\varepsilon(i', U) \neq \varnothing$ for every $i' \in I'_U$.

For a covering $\{U_1, \ldots, U_n\}$, the target sets compose by union: $T(i', \bigcup_\alpha U_\alpha) = \bigcup_\alpha T(i', U_\alpha)$, so the global feasibility set decomposes as
\begin{equation}\label{eq:feasibility-intersection}
A_\varepsilon\!\left(i', \textstyle\bigcup_\alpha U_\alpha\right) = \bigcap_\alpha A_\varepsilon(i', U_\alpha).
\end{equation}
Global section existence at $i'$ thus requires the intersection of the per-patch feasibility sets to be non-empty.

\begin{rem}\label{rem:eps-sheaf}
$\FF_j^\varepsilon$ is a sheaf: restriction is the identity on functions (a global $g : I' \to O'$ is either $\varepsilon$-close or not on each patch), so separation is trivial, and gluing follows because the $\varepsilon$-constraint is pointwise in $i \in I_U$.
This sheaf property relies on the stateless restriction; a presheaf of multi-state $\varepsilon$-sections would have a more complex structure.
This contrasts with the exact behavioral presheaf $\FF_j^{\mathrm{beh}}$, which is separated but not a sheaf (\S\ref{sec:separation}--\ref{sec:gluing}).
\end{rem}

\subsection{Obstruction depth}\label{subsec:helly}

At $\varepsilon = 0$, pairwise section existence implies global existence: if every pair of patches admits a common exact section, then the target sets are all singletons (the same point), so the global section exists.
At $\varepsilon > 0$, this fails: the feasibility sets $A_\varepsilon(i', U_\alpha)$ are balls whose pairwise intersections may be non-empty even when the full intersection is empty.
This is exactly the setup of Helly's theorem.

\begin{theorem}[Helly obstruction depth]\label{thm:helly}
Suppose $O' \subseteq \Rr^d$ is convex and $\rho$ is the Euclidean metric.
Let $\{U_1, \ldots, U_n\}$ be a covering with $n \geq d+1$, and fix $i' \in I'$.
If $A_\varepsilon(i', U_{\alpha_1} \cup \cdots \cup U_{\alpha_{d+1}}) \neq \varnothing$ for every $(d+1)$-element sub-covering, then $A_\varepsilon(i', U_1 \cup \cdots \cup U_n) \neq \varnothing$.

The bound $d+1$ is sharp (Example~\ref{ex:equilateral}).
\end{theorem}

\begin{proof}
Each $A_\varepsilon(i', U_\alpha) \cap O'$ is convex (intersection of closed balls with a convex set in $\Rr^d$).
By~\eqref{eq:feasibility-intersection}, $A_\varepsilon(i', \bigcup_{\alpha \in S} U_\alpha) = \bigcap_{\alpha \in S} A_\varepsilon(i', U_\alpha)$ for any subset $S$.
The hypothesis says every $(d+1)$-element sub-family of $\{A_\varepsilon(i', U_\alpha) \cap O'\}_\alpha$ has non-empty intersection.
By Helly's theorem~\cite{Helly1923} in $\Rr^d$, the whole family has non-empty intersection.
\end{proof}

\begin{ex}[Equilateral triangle]\label{ex:equilateral}
Let $O' = \Rr^2$, $I = \{a, b, c\}$, $j_I \equiv \star$ (all inputs collapse), $j_O = \id$, and $f(a) = (0,0)$, $f(b) = (2,0)$, $f(c) = (1, \sqrt{3})$ (equilateral triangle, side $2$).
Three singleton patches $U_\alpha = \{\alpha\}$.

The pairwise Chebyshev radius (circumradius of a segment) is $1$; the triple Chebyshev radius (circumradius of the triangle) is $2/\sqrt{3} \approx 1.155$.
For $\varepsilon \in (1, \, 2/\sqrt{3})$:
\begin{itemize}
\item $\varepsilon$-sections exist on every pair of patches (pairwise Chebyshev radius $= 1 \leq \varepsilon$),
\item but no $\varepsilon$-section exists on the triple ($2/\sqrt{3} > \varepsilon$).
\end{itemize}
This is a genuine $3$-fold obstruction in $\Rr^2$, not reducible to any pairwise failure.
The construction generalises: a regular $(d+1)$-simplex in $\Rr^d$ gives a $(d+1)$-fold obstruction.
\end{ex}

\subsection{Interpretability reading}\label{subsec:eps-interp}

The Helly obstruction has a direct interpretability reading: \emph{even if every pair of local explanations can be reconciled into a joint $\varepsilon$-approximate explanation, the full set may be globally irreconcilable---and the dimension of the output space determines how many explanations must be simultaneously checked.}

For \textbf{hard} $k$-class classification ($O' = \{1, \ldots, k\}$, discrete metric), the Helly number is $2$: pairwise reconcilability always suffices.
For \textbf{soft} classification ($O' = \Delta^{k-1}$, the probability simplex in $\Rr^{k-1}$), the Helly number is $k$: all $k$ local explanations must be checked simultaneously.
\emph{For $k \geq 3$, soft classifiers have strictly higher reconciliation complexity than hard classifiers.}

For unrestricted multi-state explanations with $\psi_b$ injective ($|S'| \geq |S_{U,b}|$, $\psi_b = \id_{S_{U,b}}$), the trivial partition (each state in its own class) decouples the output constraints across states.
The obstruction depth remains at most $\dim(O') + 1$ per state, with states contributing independently.
The stateless case ($|S'| = 1$) is thus the \textbf{worst case}: all states share one explanatory state, creating maximal constraint coupling.
Understanding this in depth remains future work.
\section{Discussion}\label{sec:discussion}

\paragraph{Concrete instantiation.}
Consider a neural network as a dynamical system: the state space $S_b$ is the space of hidden activations at a given layer, the dynamics $\alpha$ maps an activation--input pair to the next-layer activation and the network's output, and the judge $j = (j_I, j_O)$ projects the raw input/output spaces onto interpretable concepts.
A \emph{section} over a region of activation space is then an explanation: a simpler system $\cS'$ with interpretable interface $(I', O')$ that faithfully represents the network's dynamics through the judge.

Several interpretability methods illustrate this pattern.
\emph{Representation Engineering}~\cite{Zou_Phan_Chen_Campbell_Guo_Ren_Pan_Yin_Mazeika_Dombrowski_et_al_2023} identifies ``reading vectors'' $v \in \Rr^d$ such that the projection $v^T \cdot h$ of a hidden state $h \in S_b$ tracks a high-level concept (honesty, sentiment, factuality).
In our framework, $j_O$ projects the network's output onto the concept score, and a stateless section ($|S'| = 1$) is a function $g : I' \to O'$ predicting the concept score from interpretable inputs alone---precisely a \emph{probe}.
\emph{Automata extraction}~\cite{Zhang_Wei_Sun_2024} recovers a finite-state machine $\cS'$ simulating a transformer on a formal language: a \emph{multi-state section} where the explanatory system has $|S'| > 1$ internal states and the judge projects tokens to a finite alphabet.
\emph{Feature attribution} methods (saliency maps~\cite{Adebayo_Gilmer_Muelly_Goodfellow_Hardt_Kim_2018}, SHAP~\cite{Lundberg_Lee_2017}) assign importance scores to input features; Bilodeau et al.~\cite{Bilodeau_Jaques_Koh_Kim_2024} show that such methods can provably fail---in our language, the presheaf of feature-attribution explanations may have empty sections on certain regions.
The \textbf{judge is the central design choice}: it formalises \emph{what} ``interpretable'' means, and different judges yield different presheaves with different local-to-global properties.

\paragraph{What the results say.}
\emph{Separation} (Theorem~\ref{thm:separation}) says that if two behavioral explanations agree on every data slice, they are globally identical.
This is a non-trivial prediction: a behavioral probe that distinguishes two explanations locally will also distinguish them globally---there is no ``invisibility through aggregation.''
The \emph{failure of gluing} (Counterexample~\ref{cex:beh-gluing-fail}) says that locally consistent explanations need not assemble into a global one.
The obstruction is specific: it requires the dynamics to send states from distinct behavioral regions to a common after-state, overdetermining the explanation at the junction.
This is an analogous symptom to \emph{polysemanticity}: when a neuron responds to multiple unrelated concepts, a ``one concept per neuron'' explanation works on each concept's data region but fails to glue globally.
The mechanism differs---polysemanticity arises from linear-algebraic superposition in activation space, while the gluing obstruction is dynamical---but the local-success-global-failure pattern is the same.

The sheaf characterisation (Theorem~\ref{thm:singleton-ns}) makes this topological: in the continuous definable setting, $\FF_{j,1}^{\mathrm{ri}}$ is a sheaf iff no fiber of $j_I$ is robustly disconnected.
Concretely, if the interpretable input projection groups raw inputs into clusters, gluing holds when these clusters are topologically connected---it fails when a cluster splits into components that cannot be healed by nearby clusters.

The $\varepsilon$-relaxation (\S\ref{sec:approximate}) addresses the practical regime where explanations are approximate.
A probe whose worst-case error is at most $\varepsilon$ (in the metric on $O'$) is an $\varepsilon$-section.
(This is a uniform bound, stronger than the average-case accuracy typical in practice; the Helly bound applies to the worst-case regime.)
Theorem~\ref{thm:helly} then predicts: for regression probes with $d$-dimensional output ($d \geq 2$), checking pairwise consistency of local probes is insufficient---one must check $(d+1)$-wise consistency.
For hard classification ($O'$ discrete), pairwise checks suffice (target values must all agree); for soft $k$-class classifiers ($O' = \Delta^{k-1} \subset \Rr^{k-1}$, convex), Theorem~\ref{thm:helly} gives Helly number $k$.
This gives a concrete, testable prediction about when local interpretability audits miss global inconsistencies.

\paragraph{The role of the categorical framework.}
The results above are proved by direct topological and combinatorial arguments: separation is pointwise, the gluing obstruction is an explicit counterexample, the sheaf characterisation uses compactness and normality, and the Helly bound uses convex geometry.
The categorical framework---the site of open immersions, the presheaf hierarchy, the covering conditions---does not provide the proof techniques; it provides the \emph{questions}.
Without the sheaf-theoretic language, there is no reason to ask whether local probes glue, no systematic distinction between separation and gluing, and no framework connecting the Helly bound to the sheaf condition.
The machinery currently exceeds what the proofs use: adhesivity is invoked only for cogerm gluing (Theorem~\ref{thm:gluing-Fj}), not for the interpretability-relevant results.
We view this as characteristic of a first paper in a new direction: the framework's value lies in organising the conceptual landscape, and we expect future work (multi-state characterisations, graph locality, cohomological obstructions) to make the categorical tools load-bearing.

\paragraph{Abstract framework.}
The results factor cleanly through abstract axioms on a pair $(\CC, \mathcal{M})$ where $\CC$ is a finitely complete, extensive, regular category and $\mathcal{M} \subseteq \mathrm{Mono}(\CC)$ is a class of monomorphisms stable under pullback and composition (playing the r\^ole of ``open immersions'').
Each result uses only the axioms it needs:
\emph{separation} (Theorem~\ref{thm:separation}) requires only that the covering family is jointly epimorphic---every state--input pair is seen by some patch---and the pointwise nature of behavioral equivalence;
\emph{cogerm gluing} (Theorem~\ref{thm:gluing-Fj}) uses full adhesivity of $\CC$ (Appendix~\ref{app:adhesivity}), which holds for $\mathrm{Def}(\M)$ but not for $\mathrm{Def}_c(\M)$ (Remark~\ref{rem:m-adhesivity-fails}); extending cogerm gluing to the continuous setting---e.g.\ by restricting cogerm witnesses to $\mathcal{M}$-morphisms---remains open, but the interpretability-relevant results (separation, sheaf characterisation) do not require it;
the \emph{sheaf characterisation} of $\FF_{j,1}^{\mathrm{ri}}$ (Theorem~\ref{thm:singleton-ns}) additionally requires a topological enrichment $T : \CC \to \mathbf{Top}$ sending $\mathcal{M}$-morphisms to open embeddings, with compactness and normality of $T(I)$.
Regularity and the descent condition---finite jointly epimorphic $\mathcal{M}$-families are effective epimorphic---make the coverings into a subcanonical Grothendieck pretopology (Proposition~\ref{prop:abstract-pretopology}).

The concrete category $\mathrm{Def}(\M)$ satisfies all axioms except~(T): it is adhesive with a strict initial object (hence extensive), regular (image factorisation is definable), and $\mathcal{M}$ = definable open immersions.
This gives every result except the sheaf characterisation.
The subcategory $\mathrm{Def}_c(\M)$ additionally satisfies~(T) and~(A'), enabling the sheaf characterisation, but full adhesivity fails (Remark~\ref{rem:m-adhesivity-fails}), so cogerm gluing is not established in that setting; whether $\mathrm{Def}_c(\M)$ is regular remains to be verified (the standard image factorisation does not obviously transfer, since continuous definable bijections need not be homeomorphisms), though no result in this paper depends on it.
Open-cover pushouts in $\mathrm{Def}_c(\M)$ are simply unions of definable open subsets, and VK reduces to adhesivity of $\mathbf{Set}$ via the pasting lemma (Proposition~\ref{prop:defc-vk}); this suffices for the pretopology axioms but not for the colimit construction in cogerm gluing.
This factorisation clarifies what each result genuinely depends on and opens the framework to other base categories: finite $T_0$~spaces or $p$-adic definable sets.

\paragraph{Cohomological directions.}
When gluing fails, \v{C}ech cohomology can measure the obstruction.
The compatible family of Counterexample~\ref{cex:beh-gluing-fail} suggests a non-trivial \v{C}ech $1$-cocycle for the behavioral presheaf, recording the incompatible behavioral classes forced by the dynamics on the overlap.
Making this precise requires defining the coefficient structure: the behavioral presheaf is a presheaf of sets, not of abelian groups, requiring care with non-abelian \v{C}ech cohomology or a linearisation step.
For bounded linearised presheaves (truncating the target state space to $|S'| \leq N$), we expect $\check{H}^1 = 0$ on the covering of Counterexample~\ref{cex:beh-gluing-fail} for $N \geq 2$: informally, restriction maps are surjective because a fresh state can always accommodate a new source state, and surjections of sets induce surjections of free vector spaces. The gluing obstruction would then move to $\check{H}^0$, measuring compatible local data that fail to assemble globally; making this precise is a direction for future work.

\paragraph{Beyond this paper.}
The site we construct captures \emph{data locality} (decomposing the state, input, and output spaces).
Two other dimensions of interpretability locality---\emph{graph locality} (decomposing the computational graph) and \emph{feature locality} (decomposing the activation space)---require different site structures: graph locality suggests sites of compositional decompositions (in the sense of operad-algebraic systems theory~\cite{SchultzSpivak2019, NiuSpivak2025}), while feature locality suggests quotient sites for coarsenings of the judge.
Developing these structures and studying their interactions is an important direction for future work.

\bibliographystyle{eptcs}
\bibliography{bibliography}

\appendix

\section{Abstract framework}\label{app:framework}

We give precise statements and proofs for the axiomatic framework described in \S\ref{sec:discussion}.

\begin{defi}[Admissible pair]\label{def:admissible-pair}
An \emph{admissible pair} $(\CC, \mathcal{M})$ consists of:
\begin{enumerate}
\item[\textbf{(C1)}] A finitely complete, extensive (finite coproducts are disjoint and universal), regular (every morphism factors as regular epi $+$ mono, and regular epis are stable under pullback) category $\CC$;
\item[\textbf{(C2)}] A class $\mathcal{M} \subseteq \mathrm{Mono}(\CC)$ stable under pullback and composition, containing all isomorphisms.
\end{enumerate}
An $\mathcal{M}$-\emph{covering family} of an object $B$ is a finite family $\{m_i : A_i \hookrightarrow B\}_{i=1}^n$ of $\mathcal{M}$-monomorphisms that is jointly epimorphic (i.e.\ any two morphisms $g, h : B \to C$ agreeing on every $A_i$ must be equal; in concrete categories such as $\mathrm{Def}(\M)$, this is equivalent to being jointly surjective).
\end{defi}

\begin{defi}[Modular add-ons]\label{def:addons}
Given an admissible pair $(\CC, \mathcal{M})$:
\begin{enumerate}
\item[\textbf{(C3')}] \emph{Descent:} every $\mathcal{M}$-covering family is effective epimorphic (i.e.\ $B$ is the colimit of the \v{C}ech nerve of the family).
\item[\textbf{(A')}] \emph{Open-cover VK:} for any finite family of $\mathcal{M}$-subobjects $\{m_i : A_i \hookrightarrow X\}$ of a common object $X$, the pushout $\bigcup_i A_i$ along the pairwise overlaps $A_i \times_X A_j$ exists in $\CC$ and satisfies the Van Kampen condition (pulling back along any morphism yields a pushout).
\item[\textbf{(T)}] \emph{Topological enrichment:} a faithful functor $T : \CC \to \mathbf{Top}$ sending $\mathcal{M}$-morphisms to open embeddings.
\end{enumerate}
\end{defi}

\begin{prop}[Subcanonical pretopology]\label{prop:abstract-pretopology}
If $(\CC, \mathcal{M})$ satisfies \textbf{(C1)--(C2)} and \textbf{(C3')}, then $\mathcal{M}$-covering families form a subcanonical Grothendieck pretopology on~$\CC$.
\end{prop}

\begin{proof}
We verify the pretopology axioms.

\emph{(T1) Isomorphisms cover.}
$\{\id_B\}$ is an $\mathcal{M}$-family (isos $\in \mathcal{M}$ by (C2)) and is trivially jointly epimorphic.

\emph{(T2) Stability under pullback.}
Let $\{m_i : A_i \hookrightarrow B\}$ be a covering and $f : C \to B$ a morphism.
The pullback projections $p_i : C \times_B A_i \to C$ are $\mathcal{M}$-monomorphisms (pullback of $m_i \in \mathcal{M}$, stable by~(C2)).
It remains to show $\{p_i\}$ is jointly epimorphic.
By~(C3'), $\{m_i\}$ is effective epimorphic; by extensivity, this is equivalent to the canonical map $c : \bigsqcup_i A_i \to B$ being an effective epimorphism (the \v{C}ech nerve of the family identifies with the kernel pair of~$c$).
In particular $c$ is a regular epimorphism, hence stable under pullback by regularity: $f^* c$ is a regular epi.
By extensivity again, $f^* c \cong \bigsqcup_i (C \times_B A_i) \to C$ is the canonical map of $\{p_i\}$, and a regular epi in a regular category is effective, so $\{p_i\}$ is jointly epimorphic.

\emph{(T3) Transitivity.}
$\mathcal{M}$ is closed under composition (C2); compositions of jointly epimorphic families are jointly epimorphic (if $g \circ m_i \circ n_{ij} = h \circ m_i \circ n_{ij}$ for all $i, j$, then $g \circ m_i = h \circ m_i$ by joint epimorphicity of $\{n_{ij}\}$, then $g = h$ by joint epimorphicity of $\{m_i\}$).

\emph{Subcanonicity.} Effective epimorphic families satisfy descent by definition: compatible data on the covering glues uniquely.
\end{proof}

\begin{prop}[Open-cover VK for $\mathrm{Def}_c(\M)$]\label{prop:defc-vk}
Let $\M$ be an o-minimal expansion of $(\Rr, <, +, \cdot)$, $\mathrm{Def}_c(\M)$ the category of definable sets with continuous definable maps, and $\mathcal{M} = \{$definable open immersions$\}$.
Then $(\mathrm{Def}_c(\M), \mathcal{M})$ satisfies \textbf{(A')}.
\end{prop}

\begin{proof}
Let $U_1, U_2$ be definable open subsets of a definable set $X$, with $V = U_1 \cap U_2$.
We show $P = U_1 \cup U_2$ (with the subspace topology from $X$) is the pushout of $V \hookrightarrow U_1$ and $V \hookrightarrow U_2$ in $\mathrm{Def}_c(\M)$.

\emph{Universal property.}
Given continuous definable $g_1 : U_1 \to Z$ and $g_2 : U_2 \to Z$ with $g_1|_V = g_2|_V$, define $g : P \to Z$ by $g(x) = g_i(x)$ for $x \in U_i$.
This is well-defined on $V$, continuous by the pasting lemma ($U_1$ and $U_2$ are open in~$P$), and definable: $\mathrm{graph}(g) = \mathrm{graph}(g_1) \cup \mathrm{graph}(g_2)$ is a finite union of definable sets.

\emph{Van Kampen.}
The forgetful functor $U : \mathrm{Def}_c(\M) \to \mathbf{Set}$ preserves finite limits.
The set-theoretic pushout $U(P) = U(U_1) \cup U(U_2)$ coincides with the pushout in~$\mathbf{Set}$, which satisfies VK by adhesivity of~$\mathbf{Set}$.
For any $f : W \to P$, the pullback gives definable open subsets $W_i = f^{-1}(U_i) \subseteq W$ covering $W$, with $W_1 \cap W_2 = f^{-1}(V)$.
The VK equivalence at the set level (top face is a pushout $\Leftrightarrow$ front/right faces are pullbacks) lifts to $\mathrm{Def}_c(\M)$: the unique set-theoretic mediating map is continuous (pasting lemma on the open cover $\{W_1, W_2\}$) and definable (piecewise definability on a definable open cover).

The proof handles the binary case $n = 2$; the general case of $n$ subobjects follows by induction, taking $U_1 \cup \cdots \cup U_{n-1}$ and $U_n$ as the two pieces at each step (the intermediate union is a definable open subset of~$X$, hence an $\mathcal{M}$-subobject).
\end{proof}

\begin{rem}\label{rem:m-adhesivity-fails}
Full $\mathcal{M}$-adhesivity of $\mathrm{Def}_c(\M)$ \emph{fails}: pushouts along an $\mathcal{M}$-mono $m : V \hookrightarrow U_1$ with an arbitrary continuous definable $f : V \to U_2$ need not satisfy VK.
A counterexample: $V = (-1,1)$, $U_1 = \Rr$, $U_2 = \Rr^2$, $m : V \hookrightarrow U_1$ the inclusion, $f : V \to U_2$ defined by $f(t) = (t^2 - 1, t(t^2-1))$ (nodal cubic parametrisation; $f$ is injective on~$V$ but not an open immersion).
As $t \to \pm 1$, $f(t) \to (0,0)$; in the pushout $P = U_1 \sqcup_V U_2$, any continuous map out of $P$ must send $\mathrm{can}_1(\pm 1)$ and $\mathrm{can}_2(0,0)$ to the same point, so the Hausdorff quotient identifies them.
The pullback $U_1 \times_P U_2$ is then $[-1,1]$, not $(-1,1) = V$: the pushout square is not a pullback.
Property~\textbf{(A')} avoids this by requiring \emph{both} legs to be $\mathcal{M}$-monomorphisms.
\end{rem}

\section{Adhesivity of $\Sys_{\mathrm{ho}}(I,O)$}\label{app:adhesivity}

\begin{proof}[Proof of adhesivity (Theorem~\ref{thm:adhesivity})]
Let $\CC$ be adhesive with finite products, and let $U : \Sys_{\mathrm{ho}}(I,O) \to \CC$ be the forgetful functor $(S, \alpha) \mapsto S$.

\emph{Pullbacks.}
$U$ creates pullbacks: given $(S_1, \alpha_1) \xrightarrow{f} (S_3, \alpha_3) \xleftarrow{g} (S_2, \alpha_2)$, the pullback state space $P = S_1 \times_{S_3} S_2$ in $\CC$ acquires a unique dynamics $\alpha_P : P \times I \to P \times O$ making the projections into morphisms, because $- \times I$ and $- \times O$ preserve pullbacks.

\emph{Pushouts along monomorphisms.}
Let $m : (C, \gamma) \hookrightarrow (A, \alpha_A)$ be a mono (i.e.\ $m : C \hookrightarrow A$ is mono in $\CC$) and $f : (C, \gamma) \to (B, \alpha_B)$ a morphism.
The pushout $P = A \sqcup_C B$ exists in $\CC$ by adhesivity.
The dynamics $\alpha_P : P \times I \to P \times O$ is defined using the canonical identification
\[
  P \times I \;\cong\; (A \times I) \sqcup_{C \times I} (B \times I),
\]
which follows from the Van Kampen condition applied to the cube whose bottom face is the pushout $A \sqcup_C B$ and whose vertical maps are the product projections $- \times I \to -$ (all four vertical faces are pullbacks: the pullback of a projection $\pi_Y : Y \times Z \to Y$ along any $g : X \to Y$ is $X \times Z$).
On the $A \times I$ component, $\alpha_P$ is $(\mathrm{can}_A \times \id_O) \circ \alpha_A$; on the $B \times I$ component, $\alpha_P$ is $(\mathrm{can}_B \times \id_O) \circ \alpha_B$.
Compatibility on $C \times I$ follows from the morphism conditions on $m$ and $f$; maps out of the pushout $(A \times I) \sqcup_{C \times I} (B \times I)$ are determined by their restrictions, so the dynamics is unique.

\emph{Van Kampen condition.}
Given a commutative cube in $\Sys_{\mathrm{ho}}(I,O)$ whose bottom is a pushout along a mono and whose back and left faces are pullbacks, apply $U$: by adhesivity of $\CC$, the top of the underlying cube is a pushout iff the front and right faces are pullbacks.
Since $U$ creates both pullbacks (above) and pushouts along monomorphisms---the distributivity identification determines unique dynamics on any such pushout---the equivalence lifts to $\Sys_{\mathrm{ho}}(I,O)$: the ``if'' direction holds because $U$ creates pullbacks, and the ``only if'' direction holds because a pushout in $\Sys_{\mathrm{ho}}(I,O)$ is a pushout in $\CC$ equipped with the unique compatible dynamics. \end{proof}

No exponentiability hypothesis is used: the entire argument works with the uncurried dynamics $\alpha : S \times I \to S \times O$, never forming $(-)^I$. However, assuming exponentiability, the above is a special case of the general Theorem~\ref{thm:coalg_adhesivity_body}.

\begin{proof}[Proof of Theorem~\ref{thm:coalg_adhesivity_body}]
Let $U : \Coalg(H) \to \CC$ be the forgetful functor.

\emph{Pullbacks.}
$U$ creates pullbacks: given coalgebra morphisms $f : (A, \alpha) \to (C, \gamma)$ and $g : (B, \beta) \to (C, \gamma)$, the pullback $P = A \times_C B$ in $\CC$ acquires a unique coalgebra structure $\pi : P \to H(P)$ via $H(P) \cong H(A) \times_{H(C)} H(B)$ (using that $H$ preserves pullbacks).

\emph{Monomorphisms.}
Since $U$ is faithful, it reflects monomorphisms.
Conversely, if $f : (A, \alpha) \to (B, \beta)$ is mono in $\Coalg(H)$, its kernel pair $(R, \pi_1, \pi_2)$ in $\CC$ lifts to a coalgebra kernel pair by pullback creation, so $\pi_1 = \pi_2$ gives $U(f)$ mono in $\CC$.
Thus monomorphisms in $\Coalg(H)$ correspond to monomorphisms in $\CC$.

\emph{Pushouts along monomorphisms.}
Let $m : (C, \gamma) \hookrightarrow (A, \alpha_A)$ be a mono and $f : (C, \gamma) \to (B, \alpha_B)$ a morphism.
The pushout $P = A \sqcup_C B$ exists in $\CC$ by adhesivity.
The colimit cocone $\alpha_A, \alpha_B$ induces a coalgebra structure on $P$ via the universal property: the cocone $D(j) \xrightarrow{\alpha_j} H(D(j)) \xrightarrow{H(\iota_j)} H(P)$ is compatible by the coalgebra morphism conditions on $m$ and $f$, so the universal property yields $\alpha_P : P \to H(P)$.
Uniqueness: a map out of the pushout is determined by its restrictions to $A$ and $B$, and the coalgebra morphism conditions determine these.

\emph{Van Kampen condition.}
Given a commutative cube in $\Coalg(H)$ whose bottom is a pushout along a mono and whose back and left faces are pullbacks, apply $U$: by adhesivity of $\CC$, the top of the underlying cube is a pushout iff the front and right faces are pullbacks.
This equivalence lifts to $\Coalg(H)$: the ``if'' direction uses pullback creation, and the ``only if'' direction uses that a pushout in $\Coalg(H)$ maps to a pushout in $\CC$ equipped with the unique compatible coalgebra structure.\end{proof}

For $\CC = \Set$, this result was established by Padberg~\cite{Padberg2017} in the framework of $\mathcal{M}$-adhesive categories; the general statement for an arbitrary adhesive base appears to be new.
Whether the pullback-preservation properties we used can be weakened to weak pullbacks---which would cover many more functors of coalgebraic interest~\cite{Johnstone2001}---remains open.

\section{Proof of Theorem~\ref{thm:gluing-Fj}}\label{app:gluing-proof}

\begin{proof}
The proof uses adhesivity of $\mathrm{Def}(\M)$ (\S\ref{subsec:o-minimal}) and the fact that all systems in the construction share the interface $(I', O')$.

\emph{Step 1: Choose representatives and overlap witnesses.}
Choose representatives $(\cS_\alpha', \psi_\alpha)$ from each equivalence class.
Compatibility on overlaps means: for each pair $(\alpha, \beta)$, the restrictions $(\cS_\alpha', \psi_\alpha \circ n_\alpha)$ and $(\cS_\beta', \psi_\beta \circ n_\beta)$ are cogerm-equivalent in $\FF_j(\cS_\alpha \cap \cS_\beta)$, where $n_\alpha : \cS_\alpha \cap \cS_\beta \hookrightarrow \cS_\alpha$.
So there exists a system $\cS_{\alpha\beta}''$ with interface $(I', O')$ and monomorphisms
$i_\alpha : \cS_{\alpha\beta}'' \hookrightarrow \cS_\alpha'$, $i_\beta : \cS_{\alpha\beta}'' \hookrightarrow \cS_\beta'$
with identity interface components, and a morphism $\phi_{\alpha\beta} : \cS_\alpha \cap \cS_\beta \to \cS_{\alpha\beta}''$ with $i_\alpha \circ \phi_{\alpha\beta} = \psi_\alpha \circ n_\alpha$ and $i_\beta \circ \phi_{\alpha\beta} = \psi_\beta \circ n_\beta$.

\emph{Step 2: Construct the global target by colimit.}
Form the system $\cS'$ as the colimit of the diagram
\[
\cdots \cS_\alpha' \xleftarrow{i_\alpha} \cS_{\alpha\beta}'' \xrightarrow{i_\beta} \cS_\beta' \cdots
\]
in $\Sys$.
Concretely, $S_b' = \colim(S_{\alpha,b}' \leftarrow S_{\alpha\beta,b}'' \rightarrow S_{\beta,b}')$, and similarly for $S_a'$; since all $\cS_\alpha'$ and $\cS_{\alpha\beta}''$ are homogeneous and all morphisms lie in $\Sys_{\mathrm{ho}}(I',O')$, the two colimit diagrams coincide, so $S_b' = S_a'$ and $\cS'$ is homogeneous.
The interface of $\cS'$ is $(I', O')$ since all pieces share this interface and all identifications are identity on interface.
The colimit exists in $\mathrm{Def}(\M)$: it is the quotient of $\bigsqcup_\alpha S_{\alpha,b}'$ by the definable equivalence relation generated by $s \sim s'$ whenever $s = (i_\alpha)_b(t)$ and $s' = (i_\beta)_b(t)$ for some $t \in S_{\alpha\beta,b}''$.
Dynamics compatibility propagates through the generators of this equivalence relation: if $s$ and $s'$ are identified via $t \in S_{\alpha\beta,b}''$, their images under the dynamics are identified via the same $t$ (since $\cS_{\alpha\beta}''$ is a subsystem), so the relation is compatible with the dynamics by induction on chain length.

Since the covering is finite, this colimit is constructed as iterated binary pushouts along monomorphisms.
By adhesivity, each binary pushout along a mono preserves monos, so inductively the canonical maps $\mathrm{can}_\alpha : \cS_\alpha' \hookrightarrow \cS'$ are monomorphisms with identity interface components.

The dynamics $\alpha' : S_b' \times I' \to S_a' \times O'$ is well-defined on the quotient: if $s = (i_\alpha)_b(t)$ is identified with $s' = (i_\beta)_b(t)$ for $t \in S_{\alpha\beta,b}''$, then $\alpha_\alpha'(s, i')$ and $\alpha_\beta'(s', i')$ are identified in $S_a'$ by the same equivalence relation, since $\cS_{\alpha\beta}''$ is a subsystem of both $\cS_\alpha'$ and $\cS_\beta'$.

\emph{Step 3: Glue the morphisms.}
The compositions $\mathrm{can}_\alpha \circ \psi_\alpha : \cS_\alpha \to \cS'$ form a compatible family: for $x \in \cS_\alpha \cap \cS_\beta$,
\[
(\mathrm{can}_\alpha \circ \psi_\alpha)(n_\alpha(x)) = \mathrm{can}_\alpha(i_\alpha(\phi_{\alpha\beta}(x))) = \mathrm{can}_\beta(i_\beta(\phi_{\alpha\beta}(x))) = (\mathrm{can}_\beta \circ \psi_\beta)(n_\beta(x)),
\]
where the middle equality holds by the defining property of the colimit.
Since the covering is jointly surjective on $S_b \times I$ and $S_a \times O$, the compatible family defines a unique morphism $\psi : \cS \to \cS'$: compatibility ensures well-definedness on overlaps, and the dynamics square holds globally because it holds on each patch.

\emph{Step 4: Verify interface and restriction.}
On each $\cS_\alpha$, $\psi_I|_{I_\alpha} = (\mathrm{can}_\alpha)_I \circ (j_I \circ m_{I,\alpha}) = j_I \circ m_{I,\alpha}$ since $\mathrm{can}_\alpha$ has identity interface.
Joint surjectivity gives $\psi_I = j_I$; similarly $\psi_O = j_O$.
Hence $(\cS', \psi) \in \FF_j(\cS)$.

The restriction to $\cS_\alpha$ satisfies $\psi \circ m_\alpha = \mathrm{can}_\alpha \circ \psi_\alpha$.
Since $\mathrm{can}_\alpha : \cS_\alpha' \hookrightarrow \cS'$ is a monomorphism with identity interface, taking $\cS'' = \cS_\alpha'$, $i_1 = \mathrm{can}_\alpha$, $i_2 = \id$, $\phi = \psi_\alpha$ witnesses $(\cS', \mathrm{can}_\alpha \circ \psi_\alpha) \sim (\cS_\alpha', \psi_\alpha)$ in $\FF_j(\cS_\alpha)$.
\end{proof}

\section{Proof of Theorem~\ref{thm:singleton-ns}}\label{app:ns-proof}

\begin{proof}
\emph{($\Rightarrow$): robustly disconnected $\Rightarrow$ not a sheaf.}
Suppose $j_I^{-1}(t_0)$ is robustly disconnected, with $j_I^{-1}(N) \subseteq V \cup W$.
Pick a closed definable $\bar{N}_0 \subset N$ with $t_0 \in \mathrm{int}(\bar{N}_0)$, and set $C := I \setminus j_I^{-1}(\bar{N}_0)$.
Since $\mathrm{im}(j_O)$ is definably connected with $\geq 2$ points, some definably connected component $C_0$ of $O$ has $j_O|_{C_0}$ non-constant.
Pick $o_0, o_1 \in C_0$ with $j_O(o_0) \neq j_O(o_1)$, connected by a definable arc $\tilde{\gamma} : [0,1] \to C_0$ with $\tilde{\gamma}(0) = o_0$, $\tilde{\gamma}(1) = o_1$, and let $\beta : \bar{N}_0 \to [0,1]$ be a definable bump with $\beta|_{\partial \bar{N}_0} = 0$, $\beta(t_0) = 1$ (existence: normality of compact definable sets, cf.~\cite[Ch.~6]{vandenDries1998}).
Define $f : I \to O$ by $f(x) = o_0$ on $V \cap j_I^{-1}(\bar{N}_0)$ and $I \setminus j_I^{-1}(\bar{N}_0)$, and $f(x) = \tilde{\gamma}(\beta(j_I(x)))$ on $W \cap j_I^{-1}(\bar{N}_0)$.
Then $f$ is continuous (the $V$ and $W$ parts are clopen in $j_I^{-1}(\bar{N}_0)$, and $f \to o_0$ at the boundary).
On $U_1$, $f$ is constant $o_0$, so $j_O \circ f|_{U_1}$ factors through $j_I$.
On $U_2$, the $W \cap j_I^{-1}(\bar{N}_0)$ part depends on $j_I(x)$ only, the $C$ part is constant, and no point in $W \cap j_I^{-1}(\bar{N}_0)$ shares a $j_I$-value with any point in $C$, so factorisation holds.
The covering $U_1 := V \cup C$, $U_2 := W \cup C$ gives two sections $j_O \circ f|_{U_i}$ that agree on $U_1 \cap U_2 = C$ but satisfy $j_O(f(a)) = j_O(o_0) \neq j_O(o_1) = j_O(f(b))$ for $a \in V \cap j_I^{-1}(t_0)$, $b \in W \cap j_I^{-1}(t_0)$.

\emph{($\Leftarrow$): no robust disconnection $\Rightarrow$ sheaf.}
Suppose $h : I \to O'$ is locally $\ker(j_I)$-invariant on a covering $\{I_\alpha\}$ but not globally: there exist $x, y$ with $j_I(x) = j_I(y) = t_0$ and $h(x) \neq h(y)$.
Since $h$ is constant on $j_I^{-1}(t_0) \cap I_\alpha$ for each $\alpha$, the function $h|_{j_I^{-1}(t_0)}$ is locally constant with clopen level sets.
Set $A := \{z \in j_I^{-1}(t_0) : h(z) = h(x)\}$ and $B := j_I^{-1}(t_0) \setminus A$ (both non-empty, disjoint, closed in $I$).
By normality of $I$ (compact Hausdorff), there exist disjoint open $V \supseteq A$, $W \supseteq B$.
Since $S := I \setminus (V \cup W)$ is closed with $S \cap j_I^{-1}(t_0) = \varnothing$, compactness gives $t_0 \notin j_I(S)$, so some open $N \ni t_0$ satisfies $j_I^{-1}(N) \subseteq V \cup W$---a robust disconnection.
Since every locally $\ker(j_I)$-invariant function that is not globally invariant yields a robust disconnection, the contrapositive gives: no robust disconnection implies every locally invariant function is globally invariant, i.e.\ $\FF_{j,1}^{\mathrm{ri}}$ is a sheaf.
\end{proof}

\section{Counterexample Diagrams}\label{app:diagrams}

\begin{figure}[h]
\centering
\begin{tikzpicture}[
  state/.style={circle, draw, minimum size=7mm, inner sep=1pt, font=\small},
  >=stealth, auto, font=\scriptsize
]
\node[state] (s1) at (0,0) {$s_1$};
\node[state] (s2) at (2.5,0) {$s_2$};

\draw[->] (s1) to[loop above, looseness=5] node{$a/0$} (s1);
\draw[->] (s1) to[bend left=15] node{$b/0$} (s2);
\draw[->] (s2) to[bend left=15] node{$a/1$} (s1);
\draw[->] (s2) to[loop above, looseness=5] node{$b/1$} (s2);
\end{tikzpicture}
\caption{The $2$-state system of Counterexample~\ref{cex:ri-sep-fail} (separation failure for $\FF_j^{\mathrm{ri}}$).
Edge labels: input/output.
The input-splitting covering $\cS|_{\{a\}}$, $\cS|_{\{b\}}$ restricts each patch to a single input; an alternative explanation agrees on pure-input sequences but disagrees on the mixed sequence $(a,b)$.}
\label{fig:cex-ri-sep}
\end{figure}
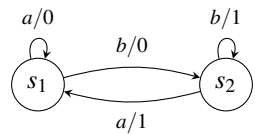

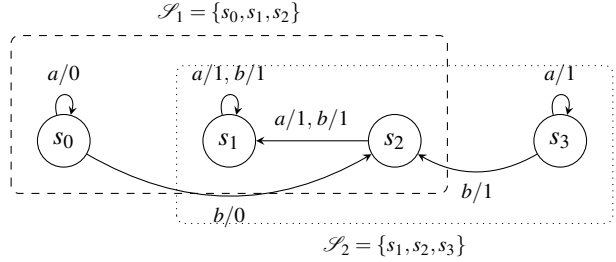
\begin{figure}[h]
\centering
\begin{tikzpicture}[
  state/.style={circle, draw, minimum size=7mm, inner sep=1pt, font=\small},
  >=stealth, auto, font=\scriptsize
]
\node[state] (s0) at (0,0) {$s_0$};
\node[state] (s1) at (2.2,0) {$s_1$};
\node[state] (s2) at (4.4,0) {$s_2$};
\node[state] (s3) at (6.6,0) {$s_3$};

\draw[->] (s0) to[loop above, looseness=5] node{$a/0$} (s0);
\draw[->] (s1) to[loop above, looseness=5] node{$a/1,\, b/1$} (s1);
\draw[->] (s3) to[loop above, looseness=5] node{$a/1$} (s3);

\draw[->] (s0) to[bend right=30] node[below]{$b/0$} (s2);
\draw[->] (s2) to node[above]{$a/1,\, b/1$} (s1);
\draw[->] (s3) to[bend left=30] node[below]{$b/1$} (s2);

\draw[dashed, rounded corners=3pt] (-0.7,-0.7) rectangle (5.1,1.4);
\node[font=\scriptsize] at (2.2,1.7) {$\cS_1 = \{s_0, s_1, s_2\}$};

\draw[dotted, rounded corners=3pt] (1.5,-1.1) rectangle (7.3,1.0);
\node[font=\scriptsize] at (4.4,-1.4) {$\cS_2 = \{s_1, s_2, s_3\}$};
\end{tikzpicture}
\caption{The $4$-state system of Counterexample~\ref{cex:beh-gluing-fail} (gluing failure for $\FF_j^{\mathrm{beh}}$).
Edge labels: input/output; the judge collapses inputs ($j_I : \{a,b\} \to \{\bullet\}$).
The dashed and dotted boxes show the covering patches; the overlap is $\{s_1, s_2\}$.
The obstruction: $s_0 \xrightarrow{b} s_2$ and $s_3 \xrightarrow{b} s_2$ force incompatible behavioral requirements on the after-state $s_2$.}
\label{fig:cex-beh-gluing}
\end{figure}

\section{Index of Notation}\label{app:notation}

\noindent
\begin{tabular}{@{}l@{\quad}l@{}}
\multicolumn{2}{@{}l}{\emph{Categories and spaces}} \\
$\CC$ & Base category (typically adhesive with finite products) \\
$\mathrm{Def}(\M)$ & Category of definable sets in the o-minimal structure $\M$ \\
$\mathrm{Def}_c(\M)$ & Wide subcategory of $\mathrm{Def}(\M)$: continuous definable maps \\
$\Coalg(H)$ & Category of coalgebras for the endofunctor $H$ \\
$\Sys$ & Category of heterogeneous Mealy machines (varying interfaces) \\
$\Sys_{\mathrm{ho}}(I,O)$ & Non-full subcategory: homogeneous systems, $f_b = f_a$ \\
$\Sys_{\mathrm{oi}}$ & Wide subcategory of $\Sys$: definable open immersions \\
$\Set$ & Category of sets \\[4pt]
\multicolumn{2}{@{}l}{\emph{Systems}} \\
$(S_b, S_a, I, O, \alpha)$ & Heterogeneous Mealy machine (system) \\
$\alpha : S_b \times I \to S_a \times O$ & Dynamics map \\
$\iota : \cS' \hookrightarrow \cS$ & Definable open immersion of systems \\[4pt]
\multicolumn{2}{@{}l}{\emph{Judge and interface}} \\
$j = (j_I, j_O)$ & Judge: morphisms to an interpretable interface $(I', O')$ \\
$I'_U$ & Restricted interface $\mathrm{im}(j_I \circ m_I)$ for subsystem $\cS_U$ \\[4pt]
\multicolumn{2}{@{}l}{\emph{Presheaves and equivalences}} \\
$\FF_j$ & Cogerm-quotiented presheaf of explanations \\
$\FF_j^{\mathrm{beh}}$ & Behavioral presheaf (quotient by $\approx$; separated) \\
$\FF_j^{\mathrm{ri}}$ & Restricted-interface presheaf \\
$\FF_{j,1}^{\mathrm{ri}}$ & Stateless subpresheaf of $\FF_j^{\mathrm{ri}}$ (sections with $|S'| = 1$) \\
$\sim$ & Cogerm equivalence on sections \\
$\approx$ & Behavioral equivalence on sections \\
\end{tabular}

\end{document}